\documentclass[12pt]{amsart}
\usepackage{latexsym,amsmath,amssymb}

\title[Homotopy groups and Lipschitz homotopy
groups]{\protect{Homotopy groups of spheres and Lipschitz homotopy
    groups of Heisenberg groups}}

\author{Piotr Haj\l{}asz, Armin Schikorra, Jeremy T. Tyson}

\address{Piotr Haj\l{}asz, Department of Mathematics, University of Pittsburgh, 301
  Thackeray Hall, Pittsburgh, PA 15260, USA, {\tt hajlasz@pitt.edu}}
\address{Armin Schikorra, Max-Planck Institut MiS Leipzig, Inselstr.
  22, 04103 Leipzig, Germany, {\tt armin.schikorra@mis.mpg.de}} 
\address{Jeremy T. Tyson, Department of Mathematics, University of
  Illinois at Urbana-Champaign, 1409 West Green Street, Urbana, IL
  61801, USA, {\tt tyson@math.uiuc.edu}} 

\thanks{P.H. was supported by NSF grant DMS-1161425. A.S. was
  supported by DAAD fellowship D/12/40670. J.T.T. was supported by NSF
grant DMS-1201875.}

plus 6pt minus 12pt
\def\eps{\varepsilon}

\def\B{{\mathbb B}}

\def\M{{\mathcal M}}
\def\N{{\mathcal N}}

\def\S{{\mathbb S}}

\newcommand{\boldg}{{\mathbf g}}

\newcommand{\cJ}{\mathcal J}

\newtheorem{theorem}{Theorem}
\newtheorem{lemma}[theorem]{Lemma}
\newtheorem{corollary}[theorem]{Corollary}
\newtheorem{proposition}[theorem]{Proposition}

\theoremstyle{definition}
\newtheorem{remark}[theorem]{Remark}
\newtheorem{definition}[theorem]{Definition}

\def\lip{{\rm Lip\,}}
\def\rank{{\rm rank\,}}
\def\supp{{\rm supp\,}}

\newcommand{\R}{\mathbb{R}}
\newcommand{\Q}{\mathbb{Q}}
\newcommand{\HI}{\mathcal{H}}
\newcommand{\Z}{\mathbb{Z}}
\newcommand{\vrac}[1]{\| #1 \|}
\newcommand{\fracm}[1]{\frac{1}{#1}}
\newcommand{\brac}[1]{\left (#1 \right )}
\newcommand{\Ep}{\bigwedge\nolimits}
\newcommand{\vol}{d{\rm vol}}

\newcommand{\barint}{
\rule[.036in]{.12in}{.009in}\kern-.16in \displaystyle\int }

\newcommand{\barcal}{\mbox{$ \rule[.036in]{.11in}{.007in}\kern-.128in\int $}}
\newcommand{\bbbr}{\mathbb R}

\newcommand{\bbbh}{\mathbb H}
\newcommand{\bbbc}{\mathbb C}

\def\mvint_#1{\mathchoice
          {\mathop{\vrule width 6pt height 3 pt depth -2.5pt
                  \kern -8pt \intop}\nolimits_{\kern -3pt #1}}%
          {\mathop{\vrule width 5pt height 3 pt depth -2.6pt
                  \kern -6pt \intop}\nolimits_{#1}}%
          {\mathop{\vrule width 5pt height 3 pt depth -2.6pt
                  \kern -6pt \intop}\nolimits_{#1}}%
          {\mathop{\vrule width 5pt height 3 pt depth -2.6pt
                  \kern -6pt \intop}\nolimits_{#1}}}

\numberwithin{theorem}{section} \numberwithin{equation}{section}

\begin{document}

\sloppy

\subjclass[2010]{Primary 53C17; Secondary 46E35, 55Q40, 55Q25}

\sloppy

%\dsp

\begin{abstract}
We provide a sufficient condition for the nontriviality of the
Lipschitz homotopy group of the Heisenberg group, $\pi_m^{\rm
Lip}(\bbbh_n)$, in terms of properties of the classical homotopy group
of the sphere, $\pi_m(\S^n)$. As an application we provide a new
simplified proof of the fact that $\pi_n^{\rm Lip}(\bbbh_n)\neq
\{0\}$, $n=1,2,\ldots$, and we prove a new result that $\pi_{4n-1}^{\rm
Lip}(\bbbh_{2n})\neq \{0\}$ for $n=1,2,\ldots$ The last result is
based on a new generalization of the Hopf invariant. We also prove
that Lipschitz mappings are not dense in the Sobolev space
$W^{1,p}(\M,\bbbh_{2n})$ when $\dim\M\geq 4n$ and $4n-1\leq p<4n$.
\end{abstract}

\maketitle

\section{Introduction}
\label{introduction}

In this paper, we provide further evidence for the role of Lipschitz
homotopy groups in the development of analysis on (non-Riemannian)
metric spaces, and specifically, in the study of Sobolev mappings with
non-Riemannian target spaces such as the sub-Riemannian Heisenberg
group. We link the study of Lipschitz homotopy groups of Heisenberg
groups with classical homotopy theory through a new notion of {\it
rank-essential} homotopy groups (Definition \ref{rank-essential}).
Using this approach, we provide new and simplified proofs of the
nontriviality of certain Lip\-schitz homotopy groups of Heisenberg
groups (previously established in \cite{BaloghF}) as well as new
examples of nontrivial Lipschitz homotopy groups. These results have
applications to the problem of density of Lipschitz mappings in Sobolev
spaces with Heisenberg targets.

The Heisenberg group $\bbbh_n$ is $\bbbr^{2n+1}$ equipped with the so
called Carnot-Carath\'eodory metric $d_{cc}$. For every compact set
$K$ there is a constant $C\geq 1$ such that $C^{-1}|x-y|\leq
d_{cc}(x,y)\leq C|x-y|^{1/2}$ for $x,y\in K$. Thus $\bbbh_n$ is
homeomorphic to $\bbbr^{2n+1}$ and the identity mapping ${\rm id}\,
:\bbbh_n\to\bbbr^{2n+1}$ is locally Lipschitz. However, the inverse
mapping ${\rm id}\, :\bbbr^{2n+1}\to\bbbh_n$ is only locally H\"older
continuous with exponent $1/2$. There is no bi-Lipschitz homeomorphism
between $\bbbh_n$ and $\bbbr^{2n+1}$, because the Hausdorff dimension
of every open set in $\bbbh_n$ is $2n+2$. The following result is well
known.

\begin{proposition}
\label{rank_n}
If $f:\bbbr^k\supset\Omega\to\bbbh_n$ is Lipschitz continuous, where
$\Omega$ is open, then it is locally Lipschitz continuous as a
mapping into $\bbbr^{2n+1}$. Hence $f$ is differentiable a.e.
%It turns out that
Moreover, $\rank df\leq n$ a.e.
\end{proposition}

Since $\bbbh_n$ is homeomorphic to $\bbbr^{2n+1}$, all of its homotopy
groups are trivial. On the other hand the Heisenberg group, as an
object of study from the viewpoint of geometric analysis on metric
spaces, is naturally equipped with its Carnot-Carath\'eodory metric
$d_{cc}$ (or other metrics bi-Lipschitz equivalent to $d_{cc}$).
As observed above, the Euclidean metric is not of this type. In the
framework of analysis on metric spaces it is natural to consider {\it
Lipschitz homotopy groups}, which are only insensitive to bi-Lipschitz
deformation. The Lipschitz homotopy groups $\pi_n^\lip(X)$ of a metric
space $X$ are defined in the same way as the classical homotopy groups
with the difference that now both mappings and homotopies between them
are required to be Lipschitz.

In the case of Riemannian manifolds homotopy groups and Lipschitz
homotopy groups are the same since continuous mappings can be smoothly
approximated. However for non-smooth spaces they may differ. The
Heisenberg group is an example since its $n$th Lipschitz homotopy
group $\pi^{\rm Lip}_n(\bbbh_n)$ is non-trivial, \cite{BaloghF}.
However, $\pi^{\lip}_m(\bbbh_n)= \{0\}$ for all $1\leq m<n$,
\cite{WengerY1}, and $\pi_m^{\lip}(\bbbh_1)=\{ 0\}$ for all $m\geq 2$,
\cite{WengerY2}. The results from \cite{BaloghF,WengerY1} stated here
did not use the language of Lipschitz homotopy groups, but they were
translated into that language in \cite{Heisenberg}.
These results show an analogy between the Lipschitz homotopy groups of
$\bbbh_n$ and the homotopy groups of the sphere $\S^n$.
The nontriviality of $\pi^\lip_n(\bbbh_n)$ is based on the following
fact (see \cite[Section~4]{BaloghF}, \cite[Theorem~3.2]{Heisenberg},
\cite[Example~3.1]{sullivan}).

\begin{proposition}
\label{sul}
There is a bi-Lipschitz embedding
% of the sphere
$\phi:\S^n\to\bbbh_n$ of the sphere $\S^n$ which is smooth as a
mapping to $\bbbr^{2n+1}$.
\end{proposition}

It was proved in \cite{BaloghF} (see also \cite{Heisenberg}) that such
an embedding cannot be extended to a  Lipschitz map
$\Phi:\B^{n+1}\to\bbbh_n$. 
Another simpler proof of this fact is provided below.  See the proof
that $\pi_n(\S^n)$ is rank-essential later in this section. Thus
$\pi^{\rm Lip}_n(\bbbh_n)\neq \{0\}$. To emphasize the analogy between
$\pi_m(\S^n)$ and $\pi_m^{\rm Lip}(\bbbh_n)$ it was asked in
\cite[Question~4.16]{Heisenberg} whether any bi-Lipschitz embedding
$\phi:\S^n\to\bbbh_n$ induces an injective homomorphism
$\pi_m(\S^n)\to \pi_m^{\rm Lip}(\bbbh_n)$. Actually the authors of the
question expected that if a smooth map $f:\S^m\to \S^n$ is not
homotopic to a constant map, $0\neq [f]\in \pi_m(\S^n)$, then the map
$g=\phi\circ f:\S^m\to \bbbh_n$ cannot be extended to a Lipschitz map
$G:\B^{m+1}\to \bbbh_n$. As will be explained below there were
strong reasons based on the Sard theorem to believe in this
conjecture, but surprisingly the conjecture is false! 
%Namely, 
Recently, Wenger and Young \cite[Theorem~1]{WengerY2} proved the
following result.

\begin{theorem}
\label{WY}
If $\alpha:\S^n\to\bbbh_n$ and $\beta:\S^m\to \S^n$ are Lipschitz and
$n+2\leq m<2n-1$, then the map $g=\alpha\circ\beta:\S^m\to\bbbh_n$ can
be extended to a Lipschitz map $G:\B^{m+1}\to\bbbh_n$.
\end{theorem}

In particular $\pi_7(\S^5)=\Z_2$, so there is a smooth map $f:\S^7\to
\S^5$ that is not homotopic to a constant map, but if $\phi:\S^5\to
\bbbh_5$ is a bi-Lipschitz embedding, then  $g=\phi\circ f:\S^7\to
\bbbh_5$ has a Lipschitz extension $G:\B^{8}\to \bbbh_5$. This is just
one example, but the above theorem leads to many more examples. It
just suffices to look at the table of the homotopy groups of the
spheres to find cases when $\pi_m(\S^n)\neq \{0\}$ and $n+2\leq
m<2n-1$. It is important to note here that it does not necessarily
imply that $\pi_m^{\rm Lip}(\bbbh_n)=\{0\}$, because in this
construction we consider mappings to $\bbbh_n$ that factor through
$\S^n$ via a bi-Lipschitz embedding into $\bbbh_n$. Perhaps there are
other mappings from $\S^m$ to $\bbbh_n$ that are not Lipschitz
homotopic to constant mappings.

\begin{definition}
\label{rank-essential}
We say that the homotopy group $\pi_m(\S^n)$ is {\em rank-essential}
if there is $f\in C^\infty(\S^m,\S^n)$ with the following property
(R): for every Lipschitz extension $F:\B^{m+1}\to\bbbr^{n+1}$,
$F|_{\partial \B^{m+1}}=f$, we have
$$
\rank dF=n+1
$$
on a set of positive measure.
\end{definition}

Clearly if $\pi_m(\S^n)$ is rank-essential, then $\pi_m(\S^n)\neq \{0\}$.
The definition is motivated by the following result.
\begin{theorem}
\label{main1}
If $\pi_m(\S^n)$ is rank-essential, then $\pi_m^{\rm Lip}(\bbbh_n)\neq \{0\}$.
\end{theorem}

\begin{proof}
Suppose to the contrary that $\pi_m(\S^n)$ is rank-essential and that
$\pi_m^{\rm Lip}(\bbbh_n)=\{0\}$. Let $f\in C^\infty(\S^m,\S^n)$ be a
mapping with property (R). Since $\pi_m^{\rm Lip}(\bbbh_n)=\{0\}$,
$g=\phi\circ f:\S^m\to\bbbh_n$ has a Lipschitz extension
$G:\B^{m+1}\to\bbbh_n$. Here $\phi:\S^n\to\bbbh_n$ is a bi-Lipschitz
embedding from Proposition~\ref{sul}. By Proposition~\ref{rank_n}
$\rank dG\leq n$ a.e., where now we regard $G$ as a mapping into
$\bbbr^{2n+1}$. The mapping $\phi^{-1}:\phi(\S^n)\to
\S^n\subset\bbbr^{n+1}$ is smooth and hence admits a smooth extension
$\Psi:\bbbr^{2n+1}\to\bbbr^{n+1}$. Clearly $F=\Psi\circ
G:\B^{m+1}\to\bbbr^{n+1}$ is a Lipschitz extension of $\Psi\circ
G|_{\partial \B^{m+1}}=f$. Since $\rank dG\leq n$ a.e., we conclude
that $\rank dF=\rank d(\Psi\circ G)\leq n$, which contradicts property
(R) of $f$.
\end{proof}

From Theorem~\ref{WY} and the proof of
Theorem~\ref{main1} we obtain the following corollary.
(See \cite[Theorem~2]{WengerY2} for a stronger statement where it is
shown that the corollary is true also for $m=n+1$.)
\begin{corollary}
\label{WY2}
If $n+2\leq m<2n-1$, then $\pi_m(\S^n)$ is not rank-essential.
\end{corollary}
In particular if $n+2\leq m<2n-1$ and $\pi_m(\S^n)\neq \{ 0\}$
(for example $\pi_7(\S^5)=\Z_2$), then every smooth mapping
$f:\S^m\to \S^n$ such that $[f]\neq 0$ admits a Lipschitz extension 
$F:\B^{m+1}\to\bbbr^{n+1}$ with $\rank dF\leq n$ a.e., despite the
fact that the image of $F$ contains the unit $(n+1)$-dimensional ball.
Indeed, otherwise we could pick a point in $\B^{n+1}\setminus
F(\B^{n+1})$ and retract $F$ onto $\S^{n}$.

The main result of the paper reads as follows.
\begin{theorem}
\label{main2}
The homotopy groups $\pi_n(\S^n)$ and $\pi_{4n-1}(\S^{2n})$ are rank-essential
for $n=1,2,3,\ldots$ and hence
$\pi_n^{\rm Lip}(\bbbh_n)\neq \{0\}$ and $\pi_{4n-1}^{\rm
  Lip}(\bbbh_{2n})\neq \{0\}$.
\end{theorem}

According to the Serre finiteness theorem \cite{Serre53} these are the
only infinite homotopy groups of spheres. The
proof of Theorem~\ref{main2} is based on differential forms. It is
done explicitly in the case of $\pi_{4n-1}^{\rm Lip}(\bbbh_{2n})$ and
implicitly in the case of $\pi_n^{\rm Lip}(\bbbh_n)$. In the latter
case we use the fact that the ball cannot be retracted to the boundary
which can be easily proved with the help of differential forms. The
language of differential forms is useful when one wants to detect the
rational homotopy groups of $CW$ complexes $\pi_m(X)\otimes\Q$. This
is the so-called rational homotopy theory discovered by Sullivan
\cite{RHT}. However in the case of spheres the rational homotopy
groups $\pi_m(\S^n)\otimes\Q$ are nontrivial exactly in the cases
covered by Theorem~\ref{main2}. This follows from the Serre finiteness
theorem. It would be very interesting to see if $\pi_m(\S^n)$ is
rank-essential for other values of $m$ and $n$.

The fact that $\pi_n^{\rm Lip}(\bbbh_n)\neq \{0\}$ was proved in
\cite{BaloghF}, but the proof presented here is different
and simpler since it does not refer to pure unrectifiability  of the
Heisenberg group, neither to the degree theory.
Another proof of an even more general result that also does not employ pure 
unrectifiability was given in \cite{RigotW}.

\begin{proof}[Proof that $\pi_n(\S^n)$ is rank-essential]
Let $f:\S^n\to \S^n$ be the identity map. If $F:\B^{n+1}\to\bbbr^{n+1}$
is a Lipschitz extension, then $\B^{n+1}\subset F(\B^{n+1})$.
In particular the image of $F$ has positive $(n+1)$-dimensional measure.
On the other hand for Lipschitz mappings we have (\cite[Theorem 3.2.3]{Federer})
$$
|F(\B^{n+1})|\leq \int_{\B^{n+1}} |\det dF|.
$$
Since $|F(\B^{n+1})|>0$ we deduce that $\rank dF\geq n+1$ on a set of positive measure.
Thus $\pi_n(\S^n)$ is rank-essential and hence $\pi_n^{\rm Lip}(\bbbh_n)\neq \{0\}$.
\end{proof}

In the last step of the proof we could refer to the Sard theorem for
Lipschitz mappings \cite[Theorem~7.6]{mattila} instead of the integral
inequality used above. Assuming $\rank dF\leq n$ a.e.\ we would
conclude that all points in $\B^{n+1}$ are critical and hence 
%the measure of 
$|F(\B^{n+1})|=0$.
%as the measure of the image of the critical set 
%equals zero.

We now investigate the connection to the Sard theorem in greater detail.

\begin{proposition}
Let $f\in C^\infty(\S^m,\S^n)$, $m\geq n$, $0\neq [f]\in \pi_m(\S^n)$. Let
$$
F:\B^{m+1}\to\bbbr^{n+1},\quad F|_{\partial \B^{m+1}}=f,
$$
be of class $C^{k,1}$, $k\geq \max\{m-n,1\}$. Then $\rank dF=n+1$
on an open set.
\end{proposition}
Here, by $C^{k,1}$ we denote the class of $C^k$ functions whose $k$th
order derivatives are Lipschitz continuous.

Indeed, if $\rank dF\leq n$ everywhere, then all points in $\B^{m+1}$
are critical and according to the Sard theorem, \cite{bates}, 
the measure of the set
$$
\B^{n+1}\subset F(\B^{m+1})=F({\rm Crit}\, F)
$$
equals zero which is a clear contradiction.

In view of the above discussion it would be natural to expect that if
$0\neq [f]\in \pi_m(\S^n)$ then any Lipschitz extension $F$ should
satisfy $\rank dF=n+1$ on a set of positive measure. However, the
result of Wenger and Young \cite[Theorem~2]{WengerY2} shows that this
is not always the case, see Corollary~\ref{WY2}. Their proof employs
an argument of Kaufman \cite{Kaufman}, who constructed a surprising
example of a surjective mapping $F\in C^1(\bbbr^{n+1},\bbbr^n)$ with
$\rank dF\leq 1$ everywhere.

One motivation for studying Lipschitz homotopy groups stems from the
problem of approximation of Sobolev mappings. In the classical setting
the answer to the question whether smooth or equivalently Lipschitz
mappings $\lip(\M,\N)$ between compact Riemannian manifolds are dense
in the Sobolev space of mappings $W^{1,p}(\M,\N)$ heavily depends on
the homotopy groups of $\N$, see \cite{bethuel1,bethuelz,
hajlaszMathAnn,hajlaszSobolev,hajlaszApproximation,hangl2}. Here
$\M$ may have boundary, but $\partial\N=\emptyset$. More precisely, if
$1\leq p<\dim\M$ and $\pi_{[p]}(\N)\neq \{0\}$, where $[p]$ is the
integral part of $p$, then Lipschitz maps are not dense in
$W^{1,p}(\M,\N)$. In the case of Sobolev mappings into the Heisenberg
group it appears that the density of Lipschitz mappings
$\lip(\M,\bbbh_n)$ in $W^{1,p}(\M,\bbbh_n)$, where $\M$ is a compact
Riemannian manifold with or without boundary, depends on Lipschitz
homotopy groups of $\bbbh_n$. For example it was proven in
\cite{Heisenberg} that if $\dim\M\geq n+1$ and $n\leq p<n+1$, then
Lipschitz maps $\lip(\M,\bbbh_n)$ are not dense in
$W^{1,p}(\M,\bbbh_n)$. Note that in this case
$\pi_{[p]}^{\lip}(\bbbh_n)=\pi_n^{\lip}(\bbbh_n)\neq \{ 0\}$. In this
paper we extend this result as follows.

\begin{theorem}
\label{density}
If $\M$ is a compact Riemannian manifold with or without boundary of
dimension $\dim\M\geq 4n$, then Lipschitz mappings
$\lip(\M,\bbbh_{2n})$ are not dense in $W^{1,p}(\M,\bbbh_{2n})$ when
$4n-1\leq p<4n$.
\end{theorem}

Again, according to Theorem~\ref{main2},
$\pi^{\lip}_{[p]}(\bbbh_{2n})=\pi^{\lip}_{4n-1}(\bbbh_{2n})\neq
\{0\}$. On the other hand we would like to point out that it is
possible to construct a smooth manifold $\N$ with one point
singularity such that all its Lipschitz homotopy groups are trivial,
yet Lipschitz mappings into $\N$ are not dense in the space
of Sobolev mappings into $\N$, see \cite{hajlaszs}.

The paper is organized as follows. In Section~\ref{heis} we provide a
brief introduction to the Heisenberg group and we prove
Proposition~\ref{rank_n}. This proof is well known, but we recall it
here for the sake of completeness and to see how the language of
differential forms and their weak exterior derivatives can be used.
Such an approach is an essential ingredient in this paper. In
Section~\ref{Sobolevmaps} we briefly recall the definition of Sobolev
mappings into $\bbbh_n$. In Section~\ref{forms} we collect basic
results about differential forms, DeRham cohomology and Sobolev
spaces. We use these facts to generalize in
Section~\ref{sec:hopfinvariant} the Hopf invariant to Lipschitz
mappings into Euclidean spaces whose derivative has low rank. SUch
generalization is essential for the proof of Theorem~\ref{main2} which
is done in Section~\ref{main}. Finally, in
Section~\ref{SobolevDensity} we prove Theorem~\ref{density}.

Those who are interested in the generalized Hopf invariant and its 
applications to homotopy groups of spheres may skip Sections~\ref{heis} 
and~\ref{Sobolevmaps} and read Sections~\ref{forms}-\ref{main}.
This material is of independent interest and does not involve
Heisenberg groups.

{\bf Acknowledgments:} We thank the reviewer for the careful reading of 
our manuscript and the valuable comments.

\section{The Heisenberg groups}
\label{heis}

The {\em Heisenberg group}  is a Lie group
$\bbbh_n=\bbbc^n\times\bbbr=\bbbr^{2n+1}$ equipped with the group law
$$
(z,t)*(z',t')=\left(z+z',t+t'+2\, {\rm Im}\,  \left(\sum_{j=1}^n z_j
  \overline{z_j'}\right)\right).
$$
A basis of left invariant vector fields is given by
\begin{equation*}
%\label{XY}
X_j=\frac{\partial}{\partial x_j} + 2y_j\frac{\partial}{\partial t},\
Y_j=\frac{\partial}{\partial y_j}-2x_j\frac{\partial}{\partial t},\ j=1,\ldots,n, \
\mbox{and}\
T=\frac{\partial}{\partial t}\, .
\end{equation*}
Here and in what follows we use the notation
$$
(z,t) = (z_1,\ldots,z_n,t) = (x_1,y_1,\ldots,x_n,y_n,t).
$$
The Heisenberg group is equipped with the
{\em horizontal distribution} $H\bbbh_n$, which is defined at every
point $p\in\bbbh_n$ by
$$
H_p\bbbh_n={\rm span}\, \{ X_1(p), Y_1(p), \ldots, Y_1(p), Y_n(p)\}.
$$
The distribution $H\bbbh_n$ is equipped with the left invariant metric
$\boldg$ such that the vectors $X_1(p),Y_1(p),\ldots,X_n(p),Y_n(p)$
are orthonormal at every point $p\in\bbbh_n$. An absolutely continuous
curve $\gamma:[a,b]\to\bbbh_n$ is called {\em horizontal} if
$\gamma'(s)\in H_{\gamma(s)}\bbbh_n$ for almost every $s$. The
Heisenberg group $\bbbh_n$ is equipped with the
{\em Carnot-Carath\'eodory metric} $d_{cc}$ which is defined as the
infimum of the lengths of horizontal curves connecting two given
points. The length of the curve is computed with respect to the metric
$\boldg$ on $H\bbbh_n$. It is well known that any two points in
$\bbbh_n$ can be connected by a horizontal curve and hence $d_{cc}$ is
a true metric. Actually, $d_{cc}$ is topologically equivalent to the
Euclidean metric. Moreover, for any compact set $K$ there is a
constant $C\geq 1$ such that
\begin{equation}
\label{SReq1}
C^{-1}|p-q|\leq d_{cc}(p,q)\leq C|p-q|^{1/2}
\end{equation}
for all $p,q\in K$.
In what follows $\bbbh_n$ will always be regarded as the metric space
$(\bbbh_n,d_{cc})$. It follows from \eqref{SReq1} that the identity mapping
from $\bbbh_n$ to $\bbbr^{2n+1}$ is locally Lipschitz, but its inverse is only
locally H\"older continuous with exponent $1/2$. The Hausdorff dimension of
any open set in $\bbbh_n$ equals $2n+2$ and hence $\bbbh_n$ is not bi-Lipschitz homeomorphic
to $\bbbr^{2n+1}$, not even locally.

\begin{proof}[Proof of Proposition~\ref{rank_n}.]
If $f=(f_1,g_1,\ldots,f_n,g_n,h):\Omega\to\bbbh_n$ is a Lipschitz
mapping from a domain $\Omega \subset \bbbr^k$, then it is locally
Lipschitz as a mapping into $\bbbr^{2n+1}$ and hence is differentiable
a.e. It follows that the derivative of $f$ is horizontal, i.e. $df(p)$
maps the tangent space $T_p\bbbr^k$ into the horizontal space
$H_{f(p)}\bbbh_n\subset T_{f(p)}\bbbr^{2n+1}$. Indeed, $f$ maps
straight lines into Lipschitz curves, and Lipschitz curves in
$\bbbh_n$ are horizontal, \cite[Proposition~11.4]{hajlaszk}. Thus $df$
maps vectors tangent to straight lines into vectors tangent to
horizontal curves. Hence $df(p)(T_p\bbbr^k)\subset H_{f(p)}\bbbh_n$
for a.e.\ $p\in\bbbr^k$. Let
\begin{equation}
\label{the-alpha-form}
\alpha = dt + 2 \sum_j (x_j \, dy_j - y_j \, dx_j)
\end{equation}
be the standard contact form on $\bbbr^{2n+1}$. It is easy to see that
the kernel of $\alpha(p)$, $p\in\bbbr^{2n+1}$, i.e.\ the collection of
vectors $v$ such that $\alpha(p)v=0$, coincides with the horizontal
space $H_p\bbbh_n$.
Hence horizontality of the derivative of $f$ means that
$f^*\alpha(p) = 0$ for a.e.\ $p$, i.e.
\begin{equation}
\label{contact-eq}
dh+2\sum_{j=1}^n \left( f_jdg_j-g_jdf_j\right) = 0 \quad \mbox{a.e.}
\end{equation}
Since the functions are Lipschitz continuous we can take the
distributional exterior derivative (see
Lemma~\ref{pr:weakdfastcommute}), obtaining
$$
\sum_{j=1}^n df_j\wedge dg_j=0.
$$
In other words if $\omega=\sum_j dx_j\wedge dy_j$ is a symplectic form
on $\bbbr^{2n}$ and $F=(f_1,g_1,\ldots,f_n,g_n)$ is a composition of
$f$ with the projection onto $\bbbr^{2n}$, then $F^*\omega=0$ a.e.\ as
a pointwise equality. Let $\cJ:T_q\bbbr^{2n}\to T_q\bbbr^{2n}$ be
given by
$$
\cJ\left(\sum_{j=1}^{n}\left(
a_j\frac{\partial}{\partial x_j}+b_j\frac{\partial}{\partial
  y_j}\right)\right) = \sum_{j=1}^{n}
\left( -b_j\frac{\partial}{\partial x_j}+a_j\frac{\partial}{\partial
    y_j}\right)\, .
$$
Then for any vectors $u,v\in T_q\bbbr^{2n}$ we have
$\omega(q)(u,v)=-\langle u,\cJ v\rangle$. If $f$ is differentiable at a point $p\in\bbbr^k$ and $(F^*\omega)(p)=0$, then
for any vectors $u,v\in V:=dF(p)(T_p\bbbr^k)\subset T_{F(p)}\bbbr^{2n}$ we have
$$
\omega(F(p))(u,v)=-\langle u,\cJ v\rangle =0.
$$
Thus the space $V$ is orthogonal to $\cJ V$ and hence $\dim V\leq n$.
The rows of the matrix $df$ are $\nabla f_1,\nabla g_1,\ldots,\nabla f_n,\nabla g_n,\nabla h$.
We proved that the rank of the minor formed by the first $2n$ rows is at most $n$. According to \eqref{contact-eq}
the last row is linearly dependent on the first $2n$ rows and hence $\rank df\leq n$ a.e.
The proof is complete.
\end{proof}

\section{Sobolev mappings into $\bbbh_n$}
\label{Sobolevmaps}

In this section we briefly recall the definition of the space of Sobolev mappings into $\bbbh_n$. For more
details, see \cite{Heisenberg}. If $\Omega\subset\R^m$ is open and $V$ is a Banach space, then the space of
vector valued Sobolev functions $W^{1,p}(\Omega,V)$ can be defined with the notion of Bochner integral and weak derivatives.
$W^{1,p}(\Omega,V)$ is a Banach space. Using local coordinate systems (see Section~\ref{forms} for more details) one can easily extend this definition to the case of mappings from
a compact manifold $W^{1,p}(\M,V)$. Any separable metric space and in particular the Heisenberg group $\bbbh_n$
admits an isometric embedding into $\ell^\infty$ (the Kuratowski embedding). Thus we can assume that
$\bbbh_n\subset\ell^\infty$. Then we define
$$
W^{1,p}(\M,\bbbh_n)=\{u\in W^{1,p}(\M,\ell^\infty):\,u(x)\in\bbbh_n\ \mbox{a.e.}\}\, .
$$
The space $W^{1,p}(\M,\bbbh_n)$ is equipped with the norm metric
$\rho(u,v)=\Vert u-v\Vert_{W^{1,p}}$. The question is whether Lipschitz mappings
$\lip(\M,\bbbh_n)$ form a dense subset of $W^{1,p}(\M,\bbbh_n)$,
see Theorem~\ref{density} and the discussion preceding its statement.

The following characterization of bounded Sobolev mappings into $\bbbh_n$ was proved in
\cite{capognal}, \cite[Proposition~6.8]{Heisenberg}.
\begin{proposition}
\label{Capogna-Lin}
A bounded function
$$
f=(z,t)=(x_1,y_1,\ldots,x_n,y_n,t):\Omega\to\bbbh_n
$$
lies in $W^{1,p}(\Omega,\bbbh_n)$ if and only if $f$ is an element of
the usual Sobolev space $W^{1,p}(\Omega,\R^{2n+1})$ and satisfies the
contact equation
\begin{equation*}
%\label{CE}
\nabla t = 2 \sum_{j=1}^n (y_j\nabla x_j - x_j \nabla y_j) \qquad \mbox{a.e.\ in $\Omega$.}
\end{equation*}
\end{proposition}
Thus the derivative of a Sobolev mapping 
$f=f(u_1,\ldots,u_m)$
maps the tangent space to a horizontal subspace of $\bbbh_n$. The length
of the gradient $\nabla f$ can be computed with respect to the Euclidean metric $|\nabla f|$ in $\R^{2n+1}$
or with respect to the sub-Riemannian metric in $\bbbh_n$
$$
|\nabla f|_\bbbh =
\left( \sum_{k=1}^m\left|\frac{\partial f}{\partial u_k}\right|_\bbbh^2\right)^{1/2}\, ,
$$
where $|v|_\bbbh$ stands for the length of the horizontal vector with respect to the
given metric in the horizontal distribution.
If the image of the mapping $f$ is contained in a bounded subset of $\bbbh_n$, then both lengths $|\nabla f|$ and
$|\nabla f|_\bbbh$ are comparable. The following result was proved in \cite[Theorem~1.6]{Heisenberg}.
\begin{proposition}
\label{T4}
Let $\Omega$ be a bounded domain in $\R^m$.
Suppose that
$f_k,f\in W^{1,p}(\Omega,\bbbh_n)$, $k=1,2,\ldots$, $1\leq p<\infty$,
$f_k\to f$ in $W^{1,p}(\Omega,\bbbh_n)$. Then
$$
\int_{\{f_k-f \not\in Z\}} |\nabla f_k|_\bbbh^p+|\nabla f|_\bbbh^p\to 0 \quad
\mbox{as $k\to\infty$,}
$$
where $Z$ denotes the center of $\bbbh_n$.
\end{proposition}
Recall that the center of $\bbbh_n$ is the $t$-axis
\begin{equation}
\label{center}
Z=\{(z,t)\in\bbbh_n:\, z=0\}.
\end{equation}
This result implies that on large sets the difference $f_k-f$ must belong to $Z$.
This surprisingly strong condition stems from the fact that the
Kuratowski embedding of $\bbbh_n$ into $\ell^\infty$ is highly
non-smooth. The identity map ${\rm id}\, :\bbbh_n\to\R^{2n+1}$ is
locally Lipschitz and hence if we assume in addition that mappings
$f_k,f$ are bounded, then $f,f_k\in W^{1,p}(\Omega,\R^{2n+1})$.
However, it is not obvious that the convergence $f_k\to f$ in
$W^{1,p}(\Omega,\bbbh_n)$ implies convergence in
$W^{1,p}(\Omega,\R^{2n+1})$, because in general the composition with a
Lipschitz function need not be continuous in the Sobolev norm
\cite[Theorem~1.2]{hajlaszGAFA}. However the following result is a
consequence of Proposition~\ref{T4}, see
\cite[Corollary~1.7]{Heisenberg}. 
\begin{corollary}
\label{T5}
Let $\M$ be a compact Riemannian manifold.
Suppose that $f_k,f\in W^{1,p}(\M,\bbbh_n)$, $k=1,2,\ldots,$
are uniformly bounded (i.e. the range of all the mappings
is contained in a bounded subset of $\bbbh_n$).
If $f_k\to f$ in $W^{1,p}(\M,\bbbh_n)$, then $f_k\to f$ in
$W^{1,p}(\M,\R^{2n+1})$.
\end{corollary}

We will also need the following fact \cite[Lemma~6.5]{Heisenberg}.
\begin{lemma}
\label{6.5}
Let $f,g\in W^{1,p}(\Omega,\bbbh_n)$. Let $S$ be the set of points
$p\in\Omega$ for which $f(p)-g(p)\in Z$. Then $\nabla f=\nabla g$ a.e.\ in $S$.
\end{lemma}

\section{Differential forms, Sobolev spaces, and DeRham cohomology}
\label{forms}

In this section, we recall some notation and properties of
differential forms on manifolds, towards the goal of showing that if
the De\-Rham cohomology is zero, then also the  $L^p$-De\-Rham
cohomology is zero. See Proposition~\ref{la:poincare}. This result
essentially follows from the $L^p$-Hodge decomposition in
\cite{IwScStr99,Scott95}.

Before we start, we need to fix some notation. Let $\M$ and $\N$ be
$C^\infty$-smooth oriented Riemannian manifolds with or without
boundary. The volume form will be denoted by $\vol$. For smooth
mappings $f: \M \to \N$ we let $f^\ast: C^\infty(\Ep^\ell \N) \to
C^\infty(\Ep^\ell \M)$ be the pullback of $\ell$-forms. By $d$ we
denote the derivative of smooth mappings, $d: C^\infty(\M,\N) \to
C^\infty(T\M,T\N)$, as well as the exterior derivative of
$\ell$-forms, $d: C^\infty(\Ep^\ell \M) \to C^\infty(\Ep^{\ell+1}
\M)$. The Hodge operator and the co-differential will be denoted  by
$*\omega$ and $\delta\omega$, respectively.

Any exterior $\ell$-form $\omega \in \Ep^\ell \M$ can be expressed in
local coordinates $x = (x_1,\ldots,x_k): U \subset \M \to \R^k$ by
$$
 \omega = \sum_{1 \leq i_1 < i_2 < \ldots < i_\ell \leq k}
 \omega_{i_1,i_2,\ldots i_\ell}\ dx_{i_1} \wedge \ldots \wedge
 dx_{i_\ell} \quad \mbox{in $U$}.
$$
We only consider local coordinate systems such that $x(U)=\B^k(0,1)$
(or $x(U)=\B^k(0,1)\cap\R^k_{+}$, if we are near the boundary of $\M$)
and such that they can be smoothly extended to larger domains
$V\Supset\overline{U}$. This will guarantee boundedness of derivatives of all orders.
We say that $\omega \in C^\infty(\Ep^\ell\M)$, $\lip(\Ep^\ell\M)$,
$L^p(\Ep^\ell\M)$, $W^{1,p}(\Ep^\ell\M)$ if the coefficients
$\omega_{i_1,i_2,\ldots i_\ell}\circ x^{-1}$ belong to the corresponding space
on $\B^k(0,1)$ (or $\B^k(0,1)\cap\R^k_{+})$.
Here $W^{1,p}$ is the standard Sobolev space. 
The norm in $W^{1,p}(\Ep^\ell\M)$ is defined as the sum of norms of
$\omega_{i_1,i_2,\ldots i_\ell}\circ x^{-1}$ in 
$W^{1,p}(\B^k(0,1))$ (or $W^{1,p}(\B^k(0,1)\cap\R^k_+)$) over
a finite family of coordinate systems that cover $\M$. A different choice of
a family of coordinate systems covering $\M$ will give an equivalent norm.
The expression
$C_0^\infty(\Ep^\ell\M)$ will stand for smooth $\ell$-forms with
compact support. In the case of manifolds with boundary we require the
support to be disjoint from the boundary.

We will make frequent use of the identities
\begin{equation}
\label{eq:commute:fastomegawedgeeta}
 f^\ast (\omega \wedge \eta) =  f^\ast\omega \wedge f^\ast\eta,
\end{equation}
and
\begin{equation}
\label{eq:commute:dfast}
 d (f^\ast\eta) = f^\ast (d\eta).
\end{equation}
Note that \eqref{eq:commute:fastomegawedgeeta} and
\eqref{eq:commute:dfast} also hold in a weak sense, in fact, we have
\begin{lemma}
\label{pr:weakdfastcommute}
Let $\M$ be a smooth, $k$-dimensional oriented manifold with or without boundary.
\begin{enumerate}
\item If $f\in W^{1,1}_{\rm loc}(\M,\R^m)$, then \eqref{eq:commute:fastomegawedgeeta}
holds pointwise a.e.
\item If $f\in W^{1,p}_{\rm loc}(\M,\R^m)$, $p\geq\ell+1$, $0\leq\ell\leq k-1$,
and $\eta\in C^\infty(\Ep^\ell\R^m)\cap W^{1,\infty}$ (i.e. $\eta$ and $|\nabla \eta|$ are bounded), then
\eqref{eq:commute:dfast} holds in the weak sense, i.e.
$$
\int_{\M} f^*\eta\wedge d\varphi = (-1)^{\ell+1}\int_{\M} f^*(d\eta)\wedge \varphi
$$
for all $\varphi\in C_0^\infty(\Ep^{k-\ell-1}\M)$.
\item If $\eta\in W^{1,p}_{\rm loc}(\Ep^{\ell_1}\M)$,
$\omega\in W^{1,p}_{\rm loc}(\Ep^{\ell_2}\M)$, $\ell_1+\ell_2\leq k-2$,
$p\geq 2$,
then $d(\eta\wedge d\omega)=d\eta\wedge d\omega$ weakly in the sense that
$$
\int_{\M} \eta\wedge d\omega\wedge d\varphi =
(-1)^{\ell_1+\ell_2}\int_{\M} d\eta\wedge d\omega\wedge\varphi
$$
for all $\varphi\in C_0^\infty(\Ep^{k-\ell_1-\ell_2-2}\M)$.
\end{enumerate}
\end{lemma}
\begin{remark}
In particular (1) and (2) hold under the condition that $f\in W^{1,k}_{\rm loc}$ and (3)
holds under the assumption that $\eta,\omega\in W^{1,k}_{\rm loc}$. This is what we will need later on.
\end{remark}
\begin{proof}
(1) is obvious. Regarding (2), observe that $d(f^*\eta)$ is not
necessarily well defined in the pointwise sense since $f^*\eta$ is
only in $L^{p/\ell}_{\rm loc}$. Thus, we need to interpret the
statement in the weak sense. Let $f_\eps$ be a smooth approximation of
$f$ in $W^{1,p}_{\rm loc}$. Integration by parts gives
$$
\int_{\M} f_\eps^*\eta\wedge d\varphi = (-1)^{\ell+1}\int_{\M} f_\eps^*(d\eta)\wedge\varphi
$$
and the result follows by letting $\eps\to 0$. The proof of (3) is
similar. Let $\omega_\eps$ and $\eta_\eps$ be smooth approximations of
$\omega$ and $\eta$ in $W^{1,p}_{\rm loc}$. Integration by parts gives
$$
\int_{\M} \eta_\eps\wedge d\omega_\eps\wedge d\varphi =
(-1)^{\ell_1+\ell_2}\int_{\M} d\eta_\eps\wedge d\omega_\eps\wedge\varphi
$$
and the result follows by letting $\eps\to 0$.
\end{proof}

Also, we have the following version of the fundamental lemma of the
calculus of variations.

\begin{lemma}
\label{pr:fundLaCV}
Assume $\M$ to be a smooth, $k$-dimensional oriented manifold with or without boundary,
and let $\eta \in L^1_{\rm loc}(\Ep^\ell \M)$ be such that
$$
 \int_{\M} \eta \wedge \varphi = 0
\quad
\mbox{for all $\varphi \in C_0^\infty(\Ep^{k-\ell}\M)$.}
$$
Then $\eta = 0$ almost everywhere in $\M$.
\end{lemma}

\begin{proof}
Let $U \subset \M$ be a coordinate patch with coordinate functions
$x = (x_1,\ldots,x_k): U \to \R^k$. Then
$$
 \eta = \sum_{1 \leq i_1 < \ldots < i_\ell \leq k} f_{i_1\ldots i_\ell} dx_{i_1} \wedge \ldots \wedge dx_{i_\ell} \quad \mbox{in $U$}.
$$
It suffices to show that $f_{i_1\ldots i_\ell}\circ x^{-1}=0$ a.e.\ in $x(U)$. Fix $1\leq i_1<\ldots<i_\ell\leq k$.
For a given $\psi \in C_0^\infty(x(U))$, let
$$
 \varphi := \psi \circ x\ |\det Dx|\ dx_{j_1}\wedge\ldots\wedge dx_{j_{k-\ell}},
$$
where $\{j_1,\ldots,j_{k-\ell}\}=\{1,2,\ldots,k\}\setminus \{i_1,\ldots,i_\ell\}$. Then
$\varphi \in C_0^\infty(\Ep^{k-\ell}\M)$ and consequently
$$
 0 = \int_{\M} \eta \wedge \varphi =
 \pm\int_{\M} f_{i_1\ldots i_\ell}\cdot \psi\circ x\ |\det Dx|\ \vol \equiv
 \pm\int_{x(U)} f_{i_1\ldots i_\ell}\circ x^{-1}\cdot \psi.
$$
Since $f_{i_1\ldots i_\ell}\circ x^{-1} \in L^1_{\rm loc}(x(U))$,
and the test function $\psi\in C_0^\infty(x(U))$ can be chosen
arbitrarily, the classical fundamental lemma of the calculus of
variations implies that $f_{i_1\ldots i_\ell}\circ x^{-1} = 0$ almost
everywhere in $x(U)$.
\end{proof}

We will need the following $L^p$-Hodge decomposition \cite[Proposition~6.5]{Scott95}.

\begin{lemma}[$L^p$-Hodge Decomposition]
\label{la:Hodgedecomp}
Let $\M$ be a smooth, compact, $k$-dimensional oriented manifold
without boundary and let  $\Omega\subset\M$ be an open subset. Then
for any $p\in (1,\infty)$ and any $\ell$-form $\eta\in
L^p(\Ep^\ell\Omega)$, $1\leq \ell\leq k$ there exist  $\omega_1\in
W^{1,p}(\Ep^{\ell-1}\Omega)$, $\omega_2\in
W^{1,p}(\Ep^{\ell+1}\Omega)$ such that
\begin{equation}
\label{eq:hodgedecomp}
\eta=d\omega_1+\delta\omega_2+h
\end{equation}
where $h\in C^\infty(\Ep^\ell\Omega)$ is
closed $dh=0$ and co-closed $\delta h=0$ and hence
harmonic.
\end{lemma}

In the case when $\Omega=\M$ the result was proved in
\cite[Proposition~6.5]{Scott95} and in the case of a general open set
we simply extend $\eta$ to $L^p(\Ep^\ell\M)$ by zero, apply the Hodge
decomposition on $\M$ and restrict all the resulting forms to
$\Omega$.

Note that the above result applies to the manifold $\M\times (0,1)$
since it can be isometrically embedded into $\M\times \S^1$ as an open
set. We will need this special case when we show
(Proposition~\ref{pr:HIhomotopicinv}) that the Hopf invariant is in
fact invariant under Lipschitz homotopies.

As an application of the Hodge decomposition we prove that if the
DeRham cohomology of an open set $\Omega\subset\M$ is zero, then also
the $L^p$-DeRham cohomology is zero. More precisely we will show

\begin{proposition}
\label{la:poincare}
Let $\M$ and $\Omega$ be as in Lemma~\ref{la:Hodgedecomp}. Suppose
that $H^\ell_{DR}(\Omega)=\{ 0\}$ for some $1\leq \ell\leq k$, i.e.
every smooth closed $\ell$-form on $\Omega$ is exact. Let $\eta\in
L^p(\Ep^\ell\Omega)$, $p\in (1,\infty)$ be weakly closed, i.e.
\begin{equation}
\label{eq:poincweakdeta}
\int_{\Omega} \eta\wedge d\varphi = 0
\quad
\mbox{for all $\varphi\in C_0^\infty(\Ep^{k-\ell-1}\Omega)$.}
\end{equation}
Then there exists $\omega\in W^{1,p}_{\rm loc}(\Ep^{\ell-1}\Omega)$
such that
$$
\eta=d\omega
\quad
\mbox{a.e.,}
$$
in particular, $\eta$ is exact in the weak sense.

If $\Omega = \M$ is compact without boundary, then $\omega \in
W^{1,p}(\Ep^{\ell-1}\M)$ with the estimate
\begin{equation}
\label{eq:omegapoincest}
\vrac{\omega}_{W^{1,p}(\M)}  \leq C\ \vrac{\eta}_{L^p(\M)},
\end{equation}
where the constant $C$ depends only on $\M$, $p$, and the norm in
$W^{1,p}(\M)$.
\end{proposition}
\begin{proof}
From Lemma~\ref{la:Hodgedecomp}, we obtain $\omega_1 \in W^{1,p}(\Ep^{\ell-1} \Omega)$, $\omega_2 \in W^{1,p}(\Ep^{\ell+1} \Omega)$,
$h \in C^\infty(\Ep^\ell \Omega)$, $dh = 0$, $\delta h = 0$, such that
$$
 \eta = d \omega_1 + \delta \omega_2 + h.
$$
Since $h \in C^\infty(\Ep^\ell \Omega)$, $dh = 0$, and
$H^\ell_{DR}(\Omega)=\{ 0\}$, there exists
$\omega_3 \in C^\infty(\Ep^{\ell-1} \Omega)$ such that $d\omega_3 =h$. Consequently,
\begin{equation}
\label{eq:lapoincstep1}
 \eta = d \brac{\omega_3 + \omega_1} + \delta \omega_2.
\end{equation}
Note that for any $\varphi \in C_0^\infty(\Ep^{k-\ell-1}\Omega)$, and for any $f \in W^{1,p}_{\rm loc}(\Ep^{\ell-1}\Omega)$,
by approximation
\begin{equation}
\label{932}
 \int_{\Omega} df \wedge d\varphi = \int_{\Omega} d(f \wedge d \varphi) = 0.
\end{equation}
Hence, from \eqref{eq:poincweakdeta} and \eqref{eq:lapoincstep1}
we infer that for any $\varphi \in C_0^\infty(\Ep^{k-\ell-1}\Omega)$,
$$
\int_{\Omega} \delta\omega_2\wedge d\varphi =
\int_{\Omega} d(\omega_3+\omega_1)\wedge d\varphi + \int_{\Omega}\delta\omega_2\wedge d\varphi =
\int_{\Omega} \eta\wedge d\varphi =0,
$$
i.e. $\delta\omega_2$ is weakly closed.
In particular, for any $\varphi \in C_0^\infty(\Ep^{k-\ell}\Omega)$
\begin{align*}
 \int_{\Omega} \delta \omega_2 \wedge (d\delta + \delta d)\varphi &=
 \int_{\Omega} \delta \omega_2 \wedge \delta d\varphi\\
 &=  \pm \int_{\Omega} *d* \omega_2 \wedge * d* d\varphi\\
 &= \pm \int_{\Omega} d* \omega_2 \wedge d* d\varphi = 0,
\end{align*}
where the last equality again follows from approximation and integration by parts
just like in \eqref{932}. That is, in the weak sense
$$
 \Delta \delta \omega_2 = 0,
$$
where $\Delta$ is the Laplace-Beltrami operator. Thus $\delta
\omega_2$ is actually smooth, see, e.g., \cite[Theorem 6.5]{Warner} or
(for the local version) \textsection 6.35 and Exercise 14 on p.\ 253
of \cite{Warner}. Since $\delta\omega_2$ is weakly closed and smooth,
it is closed in the usual sense $d(\delta\omega_2)=0$. Again,
$H^\ell_{DR}(\Omega)=0$ implies that there is $\omega_4\in
C^\infty(\Ep^{\ell-1}\Omega)$ such that $d\omega_4=\delta\omega_2$.
We have shown that
$$
 \eta = d \brac{\omega_3 + \omega_1 + \omega_4},
$$
and
$$
 \tilde{\omega} := \omega_3 + \omega_1 + \omega_4 \in 
 W^{1,p}(\Ep^{\ell-1}\Omega) + C^\infty(\Ep^{\ell-1}\Omega) \subset W^{1,p}_{\rm loc}(\Ep^{\ell-1} \Omega).
$$
If $\Omega \subset \M$ is any open subset, and we do not expect estimate \eqref{eq:omegapoincest}, we choose $\omega := \tilde{\omega}$.

Note however, that this choice of $\omega$ is not unique. In fact,
setting $\omega := \tilde{\omega} - \omega_5$ for any weakly closed
$\omega_5\in W^{1,p}$, we have
$$
 d\omega = d\tilde{\omega} = \eta \quad \mbox{a.e.\ in $\Omega$}.
$$
If $\Omega = \M$ is compact without boundary then $\tilde{\omega} \in
W^{1,p}_{\rm loc}(\Ep^{\ell-1}\M) = W^{1,p}(\Ep^{\ell-1}\M)$. By
\cite[Theorem 6.4]{IwScStr99} there exists a weakly closed form
$\omega_5$ such that $\omega := \tilde{\omega}-\omega_5 \in
W^{1,p}(\Ep^{\ell-1}\M)$ satisfies
$$
 \vrac{\omega}_{W^{1,p}} = \vrac{\tilde{\omega}-\omega_5}_{W^{1,p}}
 \leq C\ \vrac{d\tilde{\omega}}_{L^p} = C\ \vrac{\eta}_{L^p}.
$$
This concludes the proof of Proposition~\ref{la:poincare}.
\end{proof}

\section{Hopf invariant for low-rank mappings}
\label{sec:hopfinvariant}

Let $\alpha$ be the volume form on $\S^{2n}$. Then for any smooth
mapping $f:\S^{4n-1}\to \S^{2n}$ we have that
$d(f^*\alpha)=f^*(d\alpha)=0$, so $f^*\alpha=d\omega$ for some smooth
$2n-1$ form $\omega$, because $H^{2n}_{DR}(\S^{4n-1})=\{ 0\}$. The
classical Hopf invariant of $f$ is defined via the Whitehead formula
\begin{equation}
\label{eq:classicHopfinv}
\HI f=\int_{\S^{4n-1}}\omega\wedge d\omega.
\end{equation}
See \cite{BT82} for details and basic properties.

Hopf \cite[Satz II, Satz II']{Hopf} proved the following important result.
\begin{lemma}
\label{la:hopffibration}
For any $n \in \mathbb{N}$ there exists a smooth map
$f: \S^{4n-1} \to \S^{2n}$ such that $\HI f \neq 0$.
\end{lemma}

In this section we will generalize the Hopf invariant to Lipschitz
mappings $f: \S^{4n-1} \to \R^{m}$, $m \geq 2n+1$, with $\rank df \leq
2n$ almost everywhere. Let us first give the construction for smooth
$f$, $\rank df \leq 2n$. Let $\alpha$ be any smooth $2n$-form in
$\R^m$. Since $\rank df \leq 2n$ and $d\alpha$ is a $(2n+1)$-form, we
have
\begin{equation}
\label{eq:dfastalphaeq0}
d (f^\ast \alpha) = f^\ast (d\alpha) = 0,
\end{equation}
because the determinant of every ($2n+1$)-dimensional minor of $df$ has to be zero.
Thus there exists a $(2n-1)$-form $\omega$, such that
\begin{equation}
\label{eq:domegaeqfastalpha}
 d\omega = f^\ast \alpha.
\end{equation}
The Hopf invariant of $f$ is defined by
\begin{equation}
\label{eq:def:hopfinvariant}
\HI_\alpha f := \int_{\S^{4n-1}} \omega \wedge d\omega.
\end{equation}
It depends on $\alpha$, but we will show
that $\HI_\alpha f$ is independent of the particular choice of $\omega$,
and that it is actually invariant under Lipschitz homotopies with rank of the derivative
less than or equal $2n$. Obviously, if $f$ is a constant map, then $\HI_\alpha f = 0$. Moreover,

\begin{proposition}
\label{pr:classicVSthisHI}
Let $\S^{2n}$ be isometrically embedded into $\R^m$, $m\geq 2n+1$ and let
$\alpha$ be the volume form of $\S^{2n}$ smoothly extended to $\R^m$.
Then $\HI_\alpha f = \HI f$ for any smooth $f: \S^{4n-1} \to \S^{2n}
\subset \R^m$, where $\HI f$ is the classical Hopf invariant defined
in \eqref{eq:classicHopfinv}. In particular there is a smooth map
$f:\S^{4n-1}\to\R^m$ such that $\HI_\alpha f\neq 0$.
\end{proposition}
This is obvious, since $\rank df\leq 2n$ and $f^\ast (\alpha \big
|_{\S^{2n}}) = f^\ast\alpha$. The last statement follows from
Lemma~\ref{la:hopffibration}.

\begin{remark}
\label{rem:hopfinvspheresrnot1}
Observe that the Hopf invariant $\HI_\alpha f$ is defined for mappings $f: \S^{4n-1} \to \R^{m}$. 
If we denote by $\S^{4n-1}(r) = r \S^{4n-1}$ the sphere of radius $r$ centered at the origin, then 
for mappings $f: \S^{4n-1}(r) \to \R^{m}$ we set
$$
 \HI_\alpha \left(f\big |_{\S^{4n-1}(r)}\right) := \HI_\alpha (\tilde{f}_r),
$$
where $\tilde{f}_r: \S^{4n-1} \to \R^{m}$ is defined by $\tilde{f}_r(x)=f(rx)$.
\end{remark}

\subsection{Construction for Lipschitz functions}
In order to make our argument precise, we have to ensure that every
step above makes sense also for non-smooth Lipschitz mappings. For
instance, observe that  $f^\ast \alpha$ is only bounded, so one has to
interpret $d (f^\ast \alpha)$ in the weak sense.

This is a non-trivial technicality, as one cannot just approximate $f$
by smooth functions without losing the rank condition, which is
essential for the construction of $\omega$.

First, we confirm that \eqref{eq:dfastalphaeq0} holds in a weak sense.
\begin{lemma}
\label{pr:p:pullbackfrank2}
Let $m,k\geq 2n+1$.
Let $\M$ be a smooth $k$-dimensional oriented manifold with or without boundary,
and assume that $f: \M \to \R^m$ is a Lipschitz map with $\rank df \leq 2n$
almost everywhere. Then
for any smooth $2n$-form $\eta \in C^\infty(\Ep^{2n}\R^m)$,
$f^*\eta$ is weakly closed, i.e.
$$
\int_{\M} (f^\ast \eta) \wedge d\varphi = 0
\quad
\mbox{for any $\varphi \in C_0^\infty(\Ep^{k-2n-1}\M)$.}
$$
\end{lemma}
\begin{proof}
Let $\varphi \in C_0^\infty(\Ep^{k-2n-1}\M)$.
Since $f$ is Lipschitz, it is in particular in $W^{1,k}_{\rm loc}(\M,\R^m)$, so
by Lemma~\ref{pr:weakdfastcommute} we have
\begin{equation}
\label{2013}
 \int_{\M}  (f^\ast \eta)\wedge d\varphi  = -\int_{\M} f^\ast (d\eta)\wedge\varphi =0.
\end{equation}
In Lemma~\ref{pr:weakdfastcommute} we required that $\eta\in C^\infty\cap W^{1,\infty}$.
However, $W^{1,\infty}$ regularity of $\eta$ is not needed here, because the image of $f$ restricted to the support
of $\varphi$ is compact.
The last equality in \eqref{2013} follows from the fact that ${\rm rank}\, df\leq 2n$ a.e.\ and hence $f^*(d\eta)=0$ a.e.
\end{proof}

Let $\alpha$ be any smooth $2n$-form on $\R^m$, $m\geq 2n+1$ and let
$f:\S^{4n-1}\to\R^m$ be Lipschitz with $\rank df\leq 2n$ a.e.
According to Lemma~\ref{pr:p:pullbackfrank2}, $f^*\alpha$ is weakly
closed. Since $f^*\alpha\in L^2(\Ep^{2n}\S^{4n-1})$,
Proposition~\ref{la:poincare} and the fact that
$H^{2n}_{DR}(\S^{4n-1})=\{ 0\}$ imply that there is  $\omega \in
W^{1,2}(\Ep^{2n-1} \S^{4n-1})$ such that $d\omega=f^*\alpha$. Thus,
definition \eqref{eq:def:hopfinvariant} makes sense also for Lipschitz
continuous $f$. Moreover,

\begin{proposition}
\label{pr:HIchoiceofomegairrel}
Let $\omega_1, \omega_2 \in W^{1,p}(\Ep^{2n-1}\S^{4n-1})$, for some $p
\geq 2-\frac{1}{2n}$, and assume that $d\omega_1 = d\omega_2$ almost
everywhere. Then the forms $\omega_i\wedge d\omega_i$, $i=1,2$ are
integrable and
$$
 \int_{\S^{4n-1}} \omega_1 \wedge d\omega_1 = \int_{\S^{4n-1}} \omega_2 \wedge d\omega_2.
$$
In particular, for any Lipschitz map $f: \S^{4n-1} \to \R^m$ with
$\rank df \leq 2n$ a.e., definition \eqref{eq:def:hopfinvariant} of
$\HI_\alpha f$ is independent of the choice of $\omega \in
W^{1,p}(\Ep^{2n-1}\S^{4n-1})$ with $d\omega = f^* \alpha$.
\end{proposition}

This result easily follows from a slightly more general fact.
\begin{proposition}
\label{3632}
If $\omega,\nu\in W^{1,p}(\Ep^{2n-1}\S^{4n-1})$, for some 
$p \geq 2-\frac{1}{2n}$, then
$
\omega\wedge\nu \in W^{1,1}(\Ep^{4n-2}\S^{4n-1}),
$
$
d\omega\wedge\nu, \omega\wedge d\nu \in L^1(\Ep^{4n-1}\S^{4n-1}),
$
$
d(\omega\wedge\nu) = d\omega\wedge\nu - \omega\wedge d\nu
\quad
\mbox{a.e.}
$
and
$$
\int_{\S^{4n-1}} d\omega\wedge\nu = \int_{\S^{4n-1}} \omega\wedge d\nu.
$$
\end{proposition}

Assuming for a moment the validity of Proposition~\ref{3632}, we show
how to complete the proof of
Proposition~\ref{pr:HIchoiceofomegairrel}. We have
$$
\int_{\S^{4n-1}} d\omega_1\wedge(\omega_1-\omega_2) =
\int_{\S^{4n-1}} \omega_1\wedge d(\omega_1-\omega_2) = 0,
$$
and hence,
$$
\int_{\S^{4n-1}} d\omega_1\wedge\omega_1 = 
\int_{\S^{4n-1}} d\omega_1\wedge\omega_2 =
\int_{\S^{4n-1}} d\omega_2\wedge\omega_2. 
$$
This proves Proposition~\ref{pr:HIchoiceofomegairrel}.
\qed

\begin{proof}[Proof of Proposition~\ref{3632}]
We will need the following auxiliary result.
\begin{lemma}
If $f,g\in W^{1,\frac{2k}{k+1}}(\R^k)$, then $fg\in W^{1,1}(\R^k)$.
\end{lemma}

\begin{proof}
Let $p=2k/(k+1)$. If $k\geq 2$, an easy calculation shows that the
Sobolev exponent satisfies $p^*=p/(p-1)$ and hence 
by H\"older's inequality and the Sobolev embedding theorem $fg\in L^1$.  This
is also true for $k=1$ since $W^{1,1}(\bbbr)\subset L^\infty$. Sobolev
functions are absolutely continuous on almost every line \cite[Section~4.9]{EG}.
Since the product of absolutely continuous functions is absolutely
continuous, $fg$ is also absolutely continuous on almost every line. Hence we can
compute partial derivatives
\begin{equation}
\label{io23}
\frac{\partial}{\partial x_i}(fg)=
\frac{\partial f}{\partial x_i} \, g+f\frac{\partial g}{\partial x_i} \, .
\end{equation}
Again, since $p^*=p/(p-1)$ we conclude that $\partial(fg)/\partial
x_i\in L^1$ for $i=1,2,\ldots,k$. The characterization of $W^{1,1}$ by
absolute continuity on lines \cite[Section~4.9]{EG} implies that
$fg\in W^{1,1}(\R^k)$.
\end{proof}

Let $\omega,\nu\in W^{1,p}(\Ep^{2n-1}\S^{4n-1})$. We can assume that
$p=2-\frac{1}{2n}$. If $k=4n-1$, then
$$
\frac{2k}{k+1}=2-\frac{1}{2n}.
$$
Thus applying the lemma to representations of $\omega$ and $\nu$ in
local coordinates we obtain that
$\omega\wedge\nu\in W^{1,1}(\Ep^{4n-2}\S^{4n-1})$.
The product rule \eqref{io23} yields
\begin{equation}
\label{uri78}
d(\omega\wedge\nu)=d\omega\wedge\nu - \omega\wedge d\nu
\quad
\mbox{a.e.}
\end{equation}
It follows from the Sobolev embedding theorem that
$\omega,\nu\in L^{\frac{p}{p-1}}(\Ep^{2n-1}\S^{4n-1})$.
This and the H\"older inequality imply
$$
d\omega\wedge\nu,\ \omega\wedge d\nu \in L^1(\Ep^{4n-1}\S^{4n-1}).
$$
Finally, integrating \eqref{uri78} we get
$$
\int_{\S^{4n-1}} d\omega\wedge\nu - \int_{\S^{4n-1}} \omega\wedge d\nu =
\int_{\S^{4n-1}} d(\omega\wedge\nu) = 0.
$$
The last equality uses Stokes' theorem, which holds by approximating
$\omega\wedge\nu\in W^{1,1}(\Ep^{4n-2}\S^{4n-1})$ by smooth
$(4n-2)$-forms.
\end{proof}

%\begin{corollary}
%If $\omega_1,\omega_2\in W^{1,p}(\Ep^{2n-1}\S^{4n-1})$, $p\geq
%2-\frac{1}{2n}$, then $\omega_1\wedge\omega_2\in
%W^{1,1}(\Ep^{4n-2}\S^{4n-1})$.
%\end{corollary}
%
%\begin{proof}
%If $k=4n-1$, then
%$$
%\frac{2k}{k+1}=2-\frac{1}{2n}
%$$
%and the result follows from the lemma applied to representations of
%$\omega_1$ and $\omega_2$ in local coordinates.
%\end{proof}
%
%The argument is a modification of the proof of \cite[Proposition 17.22.(a)]{BT82}.
%We can assume that $p=2-\frac{1}{2n}$.
%Let $\omega_1, \omega_2 \in W^{1,p}(\Ep^{2n-1} \S^{4n-1})$ be such that
%$$
% d(\omega_1 -\omega_2) = 0 \quad \mbox{a.e.}
%$$
%Since $\omega_1, \omega_2 \in W^{1,p}(\Ep^{2n-1}\S^{4n-1})$ with $p = 2-\frac{1}{2n}$, 
%by Sobolev embedding, $\omega_1, \omega_2 \in L^{\frac{p}{p-1}}(\Ep^{2n-1}\S^{4n-1})$, thus
%$$
% \omega_i \wedge d \omega_j \in L^1(\Ep^{4n-1}\S^{4n-1}) \qquad \mbox{for any }i,j \in \{1,2\}.
%$$
%Moreover, we have
%$$
% d \brac{(\omega_1 - \omega_2)\wedge \omega_1} = d (\omega_1 - \omega_2)\wedge \omega_1 - (\omega_1 - \omega_2)\wedge d\omega_1 = - (\omega_1 - \omega_2)\wedge d\omega_1
%$$
%almost everywhere. Hence,
%\begin{align*}
% \int_{\S^{4n-1}} \omega_1 \wedge d \omega_1 - \int_{\S^{4n-1}} \omega_2 \wedge d \omega_2
% &=  \int_{\S^{4n-1}} (\omega_1 - \omega_2)\wedge d \omega_1\\
% &= -\int_{\S^{4n-1}} d \brac{(\omega_1 - \omega_2)\wedge \omega_1}\\
% &= 0.
%\end{align*}
%The last step is Stokes' theorem, which holds obviously by approximating
%$(\omega_1 - \omega_2)\wedge \omega_1 \in W^{1,1}(\Ep^{4n-2}\S^{4n-1})$ by smooth $(4n-2)$-forms.
%\end{proof}

Next, we show that $\HI_\alpha f$ is invariant under Lipschitz
homotopies of rank at most $2n$.

\begin{proposition}
\label{pr:HIhomotopicinv}
Let $f,g: \S^{4n-1} \to \R^m$ be two Lipschitz maps of rank at most
$2n$  and assume that there is a Lipschitz homotopy
$$
H: [0,1]\times\S^{4n-1} \to \R^m, \qquad H(0,\cdot) = f(\cdot), \,
H(1,\cdot) = g(\cdot),
$$
such that ${\rm rank}\, dH\leq 2n$ a.e. Then
$$
\HI_\alpha f =\HI_\alpha g.
$$
\end{proposition}
\begin{proof}
We adapt the argument from \cite[Proposition 17.22.(c)]{BT82}.
However, since we are dealing with non-smooth mappings we have to be very careful.
We may assume that $H: [0,1]\times \S^{4n-1} \to \R^m$ is constant in $t$ for
$0\leq t \leq 1/4$ and $3/4\leq t\leq 1$. If not, we take a Lipschitz function $s(t): [0,1] \to [0,1]$ such that
$$
 s(t) := \begin{cases}
          0 \quad 0\leq t \leq \fracm{4},\\
          1 \quad \frac{3}{4}\leq t\leq 1,
         \end{cases}
$$
and consider $H(s(t),x)$ instead of $H(t,x)$, which is still Lipschitz, and also satisfies the rank condition. We have
$$
 H^\ast \alpha \in L^\infty (\Ep^{2n} ((0,1) \times \S^{4n-1})).
$$
Since $\S^{4n-1}$ is a deformation retract of $(0,1)\times\S^{4n-1}$
we conclude that $H^{2n}_{DR}((0,1)\times\S^{4n-1})=H^{2n}_{DR}(\S^{4n-1})=\{ 0\}$,
\cite[Corollary~4.1.2.2]{BT82}.
Now from the fact that $\rank dH\leq 2n$ and from
Lemma~\ref{pr:p:pullbackfrank2} we infer that $H^*\alpha$ is weakly
closed. Since $(0,1)\times\S^{4n-1}$ can be isometrically embedded
into the compact manifold $\S^1\times\S^{4n-1}$ as an open set by
Proposition~\ref{la:poincare} there is 
$\omega\in W^{1,2}_{\rm loc}(\Ep^{2n-1}(0,1)\times\S^{4n-1})$ such that
$$
 d\omega = H^\ast \alpha \quad \mbox{a.e.}
$$
Denote by
$$
\imath_t:\S^{4n-1}\to \{ t\}\times \S^{4n-1}\subset (0,1)\times\S^{4n-1}.
$$
the canonical embedding of the sphere by the identity. From the
Rademacher and Fubini theorems it follows that for almost every $t\in
(0,1)$, $H$ is differentiable at almost all points of the sphere $\{
t\}\times\S^{4n-1}$. Thus the chain rule implies that 
\begin{equation}
\label{chain}
(H\circ\imath_t)^*\alpha = \imath_t^*H^*\alpha
\quad
\mbox{a.e.\ in $\S^{4n-1}$}
\end{equation}
Note also that $\omega_t:=\imath_t^*\omega$ is defined a.e.\ on
$\S^{4n-1}$ for almost all $t\in (0,1)$. Approximate $\omega$ by
$$
\omega^\eps\in C^\infty\left(\Ep^{2n-1}(0,1)\times\S^{4n-1}\right)
\quad
\mbox{in}
\quad
W^{1,2}_{\rm loc}\left(\Ep^{2n-1}(0,1)\times\S^{4n-1}\right).
$$
It follows from Fubini's theorem (cf.\ \cite[p.~189]{hajlaszSobolev})
that there is a sequence $\eps_i\to 0$ such that 
\begin{equation}
\label{conv}
\omega^{\eps_i}_t:=\imath_t^*\omega^{\eps_i} \to \imath_t^*\omega=\omega_t
\quad
\mbox{in}
\quad
W^{1,2}\left(\Ep^{2n-1}\S^{4n-1}\right)
\end{equation}
and
$$
\imath_t^*d\omega^{\eps_i}\to \imath_t^*d\omega
\quad
\mbox{in}
\quad
L^2\left(\Ep^{2n}\S^{4n-1}\right)
$$
for almost all $t\in (0,1)$. Since
$$
\imath_t^* d\omega^{\eps_i}=d\imath_t^*\omega^{\eps_i}\to d\omega_t
\quad
\mbox{in $L^2$}
$$
we conclude that
\begin{equation}
\label{three}
d\omega_t=\imath_t^*d\omega
\quad
\mbox{a.e.\ on $\S^{4n-1}$ for a.e.\ $t\in (0,1)$.}
\end{equation}
Fix $t_0\in (0,1/4)$ and $t_1\in (3/4,1)$ such that \eqref{chain}, \eqref{conv} and \eqref{three}
are satisfied. We have
$$
d\omega_{t_0}=\imath_{t_0}^*d\omega =
\imath_{t_0}^*H^*\alpha =
(H\circ\imath_{t_0})^*\alpha = f^*\alpha
\quad
\mbox{a.e.\ in $\S^{4n-1}$.}
$$
Similarly
$$
d\omega_{t_1}=g^*\alpha
\quad
\mbox{a.e.\ in $\S^{4n-1}$.}
$$
Hence
\begin{eqnarray*}
\HI_\alpha f -\HI_\alpha g
& = &
\int_{\S^{4n-1}}\omega_{t_0}\wedge d\omega_{t_0} -
\int_{\S^{4n-1}}\omega_{t_1}\wedge d\omega_{t_1} \\
& = &
\lim_{i\to\infty}
\left(
\int_{\S^{4n-1}}\omega_{t_0}^{\eps_i}\wedge d\omega_{t_0}^{\eps_i} -
\int_{\S^{4n-1}}\omega_{t_1}^{\eps_i}\wedge d\omega_{t_1}^{\eps_i}
\right) \\
& = &
\lim_{i\to\infty}
\int_{\partial( (t_0,t_1)\times\S^{4n-1})}
\omega^{\eps_i}\wedge d\omega^{\eps_i} \\
& = &
\lim_{i\to\infty}
\int_{(t_0,t_1)\times\S^{4n-1}}
d(\omega^{\eps_i}\wedge d\omega^{\eps_i})\\
& = &
\lim_{i\to\infty}
\int_{(t_0,t_1)\times\S^{4n-1}}
d\omega^{\eps_i}\wedge d\omega^{\eps_i}\\
& = &
\int_{(t_0,t_1)\times\S^{4n-1}} d\omega\wedge d\omega \\
& = &
\int_{(t_0,t_1)\times\S^{4n-1}}
H^*\alpha\wedge H^*\alpha \\
& = &
\int_{(t_0,t_1)\times\S^{4n-1}}
H^*(\alpha\wedge\alpha)=0.
\end{eqnarray*}
The last equality follows from the fact that $\rank dH\leq 2n$ a.e.\ and
$\alpha\wedge\alpha$ is a $4n$-form, so $H^*(\alpha\wedge\alpha)=0$
a.e.
\end{proof}

We will also need the following convergence result.

\begin{proposition}
\label{pr:HIconvergence}
Let $g_k$, $g \in \lip(\S^{4n-1},\R^{m})$ be Lipschitz mappings with
$\rank dg_k$, $\rank dg \leq 2n$ almost everywhere and such that for
a given $\alpha \in C^\infty(\Ep^{2n}\R^m)$
$$
 \lim_{k \to \infty} \vrac{g_k^*\alpha - g^*\alpha}_{L^p(\Ep^{2n}\S^{4n-1})} = 0,
$$
for some $p \geq 2-\frac{1}{2n}$. Then
$$
 \lim_{k \to \infty} \HI_\alpha g_k = \HI_\alpha g.
$$
\end{proposition}
\begin{proof}
We can assume that $p=2-\frac{1}{2n}$.
According to Lemma~\ref{pr:p:pullbackfrank2} the forms $g^*\alpha$ and $g_k^*\alpha$
are weakly closed. Hence from Proposition~\ref{la:poincare} there exist
$\omega$, $\omega_k \in W^{1,p}(\Ep^{2n-1}\S^{4n-1})$ with $d\omega = g^*\alpha$, $d\omega_k = g_k^* \alpha$, and
such that
$$
 \vrac{\omega}_{L^{\frac{p}{p-1}}} \leq C \vrac{\omega}_{W^{1,p}} \leq C' \vrac{g^\ast \alpha}_{L^p},
$$
and similarly
$$
 \vrac{\omega_k}_{L^{\frac{p}{p-1}}} \leq C'  \vrac{g_k^\ast \alpha}_{L^p}.
$$
We used here the Sobolev inequality and the fact that $p^*=\frac{p}{p-1}$.
In view of Proposition~\ref{pr:HIchoiceofomegairrel},
$$
  \HI_\alpha g = \int_{\S^{4n-1}} \omega \wedge d\omega, \quad
  \HI_\alpha g_k = \int_{\S^{4n-1}} \omega_k \wedge d\omega_k. 
$$
Hence
\begin{eqnarray*}
 \HI_\alpha g_k - \HI_\alpha g
 &=&
 \int_{\S^{4n-1}} \omega_k \wedge d\omega_k - \omega \wedge d\omega\\
 &=&
 \int_{\S^{4n-1}} \omega_k \wedge (d\omega_k - d\omega) + (\omega_k- \omega) \wedge d\omega\\
 &=&
 \int_{\S^{4n-1}} \omega_k \wedge (d\omega_k - d\omega) + d(\omega_k- \omega) \wedge \omega\\
 &\leq&
 C\ \left(\vrac{\omega_k}_{L^{\frac{p}{p-1}}} + \vrac{\omega}_{L^{\frac{p}{p-1}}}\right)\ \vrac{g^* \alpha-g_k^* \alpha}_{L^p}\\
 &\leq&
 C\ (\vrac{g_k^*\alpha}_{L^{p}} + \vrac{g^*\alpha}_{L^{p}})\ \vrac{g^* \alpha-g_k^* \alpha}_{L^p}
\xrightarrow{k \to \infty} 0.
\end{eqnarray*}
The third-last equality follows from Proposition~\ref{3632}.
\end{proof}

\section{Proof of Theorem~\ref{main2}}
\label{main}

The case of $\pi_n(\S^n)$ having already been proved in Section~\ref{introduction}, it remains to show that the homotopy group $\pi_{4n-1}(\S^{2n})$ is rank-essential for $n \in \mathbb{N}$.
Let $f: \S^{4n-1} \to \S^{2n}\subset\R^{2n+1}$ be the mapping, and $\alpha$ the $2n$-form on $\R^{2n+1}$ such that
\begin{equation}
\label{eq:HIfneq0}
 \HI_\alpha f \neq 0.
\end{equation}
See Proposition~\ref{pr:classicVSthisHI}.
Assume by contradiction that $\pi_{4n-1}(\S^{2n})$ is not rank-essential. Hence
there exists a Lipschitz extension $F: \B^{4n} \to \R^{2n+1}$ such that $\rank dF \leq 2n$ almost everywhere in $\B^{4n}$.
Define the homotopy
$$
 H(t,\theta): [0,1] \times \S^{4n-1} \to \R^{2n+1}
$$
between $f=H(1,\cdot)$ and a constant map $g=H(0,\cdot)$ via
$$
 H(t,\theta) := F(t\theta).
$$
This homotopy $H$ is clearly Lipschitz, with $\rank dH \leq 2n$.
Obviously, $\rank df$ and $\rank dg$ do not exceed $2n$. Then, since
the Hopf invariant $\HI_\alpha f$ does not change under Lipschitz rank
$2n$-homotopies, see Proposition~\ref{pr:HIhomotopicinv},
$$
 \HI_\alpha f = \HI_\alpha g = 0,
$$
which contradicts \eqref{eq:HIfneq0}. The proof is complete.
\qed

\section{Proof of Theorem~\ref{density}}
\label{SobolevDensity}

The proof is similar to that of Theorem~1.2(a) and Proposition~1.3 in \cite{Heisenberg}.

Assume first that $\M=\B^{4n}$.
Let $\phi: \S^{2n} \to \bbbh_{2n}$ be a bi-Lipschitz map, which is a smooth embedding as a map from
$\S^{2n}$ to $\R^{4n+1}$, see Proposition~\ref{sul}.
Let $f_0 \in C^\infty(\S^{4n-1},\S^{2n})$ be the Hopf map from Lemma~\ref{la:hopffibration} such that
$$
 \HI f_0 \neq 0.
$$
It easily follows from Proposition~\ref{Capogna-Lin} that
$$
 f(x) := \phi\circ f_0\left (\frac{x}{|x|}\right ) \in W^{1,p}(\B^{4n},\bbbh_{2n}),
 \quad
 \mbox{for all $1 \leq p < 4n$.}
$$
We will prove that $f$ cannot be approximated in
$W^{1,p}(\B^{4n},\bbbh_{2n})$ by Lipschitz mappings $\lip(\B^{4n},\bbbh_{2n})$
when $4n-1\leq p<4n$. Suppose to the contrary that there is a sequence
$g_k \in \lip(\B^{4n},\bbbh_{2n})$ such that
$$
g_k \to f \quad \mbox{in $W^{1,p}(\B^{4n},\bbbh_{2n})$.}
$$
Note that by Proposition~\ref{rank_n} both
$\rank dg_k$ and $\rank df$ do not exceed $2n$. Formally, $f$ is not Lipschitz, but it is locally Lipschitz away from the
singularity at the origin and hence Proposition~\ref{rank_n} applies to $f$ as well.

Choose $\alpha \in C_0^\infty(\Ep^{2n}\R^{4n+1})$ to be a smooth extension of the push-forward
$\phi_* \vol_{\S^{2n}}$. Recalling our definition of the Hopf invariant of mappings whose
domains are scaled spheres $\S^{4n-1}(r)$, see Remark~\ref{rem:hopfinvspheresrnot1},
\begin{equation}
\label{eq:HIalphafonsphere}
 \HI_\alpha \left(f\big |_{\S^{4n-1}(r)}\right) =  \HI (f_0) \neq 0 \quad \mbox{for all $r \in (0,1)$}.
\end{equation}

On the other hand, $g_k \in \lip(\B^{4n},\bbbh_{2n})$,
and hence 
$g_k\big|_{\S^{4n-1}(r)}$
as a mapping to $\R^{4n+1}$ is Lipschitz homotopic to a constant map with the homotopy satisfying the
rank condition $\rank dH\leq 2n$ a.e.\ (see Proposition~\ref{rank_n}). Thus Proposition~\ref{pr:HIhomotopicinv}
yields
\begin{equation}
\label{eq:HIalphagkonsphere}
 \HI_\alpha \left(g_k \big |_{\S^{4n-1}(r)}\right) = 0 \quad \mbox{for all $k$ and all $r \in (0,1)$}.
\end{equation}
We are now going to show that \eqref{eq:HIalphafonsphere} and \eqref{eq:HIalphagkonsphere} contradict each other.

Since the mappings $g_k$ are not necessarily uniformly bounded we cannot claim that
$g_k\to f$ in $W^{1,p}(\B^{4n},\R^{4n+1})$, see Corollary~\ref{T5}. In particular we
cannot claim that $\nabla g_k\to \nabla f$ in $L^p(\B^{4n})$.
Nevertheless we can assume upon passing to a subsequence that $g_k\to f$ a.e.\ in $\B^{4n}$. We will construct sets $E_k$ such that
\begin{equation}
\label{zbieznosc}
\chi_{E_k}\nabla g_k\to \nabla f
\quad
\mbox{in $L^p(\B^{4n})$.}
\end{equation}
Let $K=\supp\alpha$, let
$$
S_k=\{x\in\B^{4n}:\, g_k(x)-f(x)\in Z\},
$$
where $Z$ is the center of $\bbbh_{2n}$ defined in \eqref{center},
and let
$$
E_k=S_k\cup g_k^{-1}(K).
$$
We claim that \eqref{zbieznosc} is true. According to Lemma~\ref{6.5},
$\nabla g_k=\nabla f$ a.e.\ in $S_k$ and hence
$$
\int_{S_k}|\nabla f-\nabla g_k|^p=0.
$$
Since the mappings $f$ and $g_k|_{g_k^{-1}(K)}$ are uniformly bounded,
the Euclidean lengths $|\nabla f|$ and $|\chi_{E_k}\nabla g_k|$
are comparable to the Heisenberg lengths $|\nabla f|_\bbbh$ and
$|\chi_{E_k}\nabla g_k|_\bbbh$ respectively on the set $\B^{4n}\setminus S_k$.
Thus Proposition~\ref{T4} yields
$$
\int_{\B^{4n}\setminus S_k} |\nabla f|^p+|\chi_{E_k}\nabla g_k|^p \leq
C\int_{\B^{4n}\setminus S_k} |\nabla f|_\bbbh^p + |\chi_{E_k}\nabla g_k|_\bbbh^p \to 0.
$$
Hence
\begin{eqnarray*}
\lefteqn{\int_{\B^{4n}} |\nabla f-\chi_{E_k}\nabla g_k|^p} \\
& \leq &
C\left( \int_{S_k} |\nabla f-\nabla g_k|^p +
\int_{\B^{4n}\setminus S_k} |\nabla f|^p+|\chi_{E_k}\nabla g_k|^p\right)\to 0.
\end{eqnarray*}
Now it follows from Fubini's theorem that, up to a subsequence which
we again denote by $g_k$,
$$
 \chi_{E_k}\nabla g_k  \big |_{\S^{4n-1}(r)} \xrightarrow{k \to
   \infty} \nabla f  \big |_{\S^{4n-1}(r)}
 \quad
 \mbox{in $L^p(\S^{4n-1}(r))$}
$$
for almost any $r \in (0,1)$. This and the almost everywhere
convergence $g_k\to f$ implies that 
\begin{equation}
\label{Iamtired}
\chi_{E_k}\left(g_k\big|_{\S^{4n-1}(r)}\right)^*\alpha \rightarrow
\left(f\big|_{\S^{4n-1}(r)}\right)^*\alpha
\quad
\mbox{in}\ L^{p/2n}\left(\Ep^{2n}\S^{4n-1}(r)\right)
\end{equation}
for almost all $r\in (0,1)$. On the other hand, since $K=\supp\alpha$,
$g_k^*\alpha = 0$ a.e.\ in $\B^{4n}\setminus g_k^{-1}(K)$ and hence
$g_k^*\alpha = 0$ a.e.\ in $\B^{4n}\setminus E_k$. Accordingly, 
$$
\left(g_k\big|_{\S^{4n-1}(r)}\right)^*\alpha =
\chi_{E_k}\left(g_k\big|_{\S^{4n-1}(r)}\right)^*\alpha \quad \mbox{for
  a.e.\ $r\in (0,1)$,}
$$
which, in conjunction with \eqref{Iamtired}, yields
$$
\lim_{k \to \infty} \left\Vert \left( g_k  \big |_{\S^{4n-1}(r)} \right)^* \alpha - \left( f  \big |_{\S^{4n-1}(r)} \right)^* \alpha\right\Vert_{L^{p/2n}(\Ep^{2n}\S^{4n-1}(r))} = 0.
$$
This, \eqref{eq:HIalphafonsphere}, and
Proposition~\ref{pr:HIconvergence} imply that for $p \geq 4n-1$ and
almost all $r\in (0,1)$, we have
$$
 \lim_{k \to \infty} \HI_\alpha \left(g_k  \big |_{\S^{4n-1}(r)}\right) =
 \HI_\alpha \left( f\big|_{\S^{4n-1}(r)}\right) \neq 0.
$$
This conclusion contradicts \eqref{eq:HIalphagkonsphere}.

If $\M$ is a general manifold of dimension $\dim\M\geq 4n$, then the
result follows from the case $\B^{4n}$ by a simple surgery  as in the
proof of Theorem~1.2 in \cite{Heisenberg}. We simply construct a
mapping $f\in W^{1,p}(\M,\bbbh_{2n})$ such that on a family of $4n$
dimensional slices in $\M$ it coincides with the mapping constructed
above. By using the Fubini theorem one easily arrives at a
contradiction by employing the case of $\B^{4n}$. \qed

\end{document}